\documentclass[runningheads,11pt]{llncs}

\usepackage{amssymb}
\usepackage{graphicx}
\usepackage{amsmath}
\usepackage{stmaryrd}
\usepackage{url}
\usepackage[color,matrix,arrow]{xy}
\usepackage{wrapfig}
\usepackage{textcomp}
\usepackage{bussproofs}
\usepackage{proof}
\usepackage{color}
\usepackage{enumerate}

\newcommand{\mb}{\mathbb}
\newcommand{\mc}{\mathcal}

\newcommand{\imp}{\rightarrow}

\newcommand{\Imp}{\Rightarrow}
\newcommand{\sub}{\subseteq}

\newcommand{\D}{\Diamond}
\newcommand{\B}{\Box}

\title{Residuated Basic Logic II.\\{\large Interpolation, Decidability and Embedding}}
\author{Minghui Ma\inst{1} \and Zhe Lin \inst{2}}
\institute{Institute for Logic and Intelligence, Southwest University,\\
Beibei District, Chongqing, 400715, China.\\
\email{mmh.thu@gmail.com}
\and Corresponding Author.\\
Institute of Logic and Cognition, Sun Yat-sen University\\
No. 135, Xingang Xi Road, Guangzhou, China\\
Faculty of Mathematics and Computer Science, Adam Mickiewicz University,
Umultowska 87, 61-614 Pozna\'{n}, Poland\\
\email{pennyshaq@gmail.com}}

\begin{document}

\maketitle

\begin{abstract}
We prove that the sequent calculus $\mathsf{L_{RBL}}$ for residuated basic logic $\mathsf{RBL}$ has strong finite model property, and that intuitionistic logic can be embedded into basic propositional logic $\mathsf{BPL}$.  Thus $\mathsf{RBL}$ is decidable. Moreover, it follows that the class of residuated basic algebras has the finite embeddability property, and that $\mathsf{BPL}$ is PSPACE-complete, and that intuitionistic logic can be embedded into the modal logic $\mathsf{K4}$.
\end{abstract}

\section{Introduction}
The first part of this paper (\cite{ML14}) developed the residuated basic logic $\mathsf{RBL}$ which is the logic of residuated basic algebras (bounded distributive lattice order residuated groupoid with weakening and restricted contraction), and we proved that $\mathsf{RBL}$ is a conservative extension of Visser's basic propositional logic $\mathsf{BPL}$. We presented the algebraic system $\mathsf{S_{RBL}}$, and its sequent calculus formalization $\mathsf{L_{RBL}}$ which has cut elimination and subformula property.

This part II aims to show that the sequent calculus $\mathsf{L_{RBL}}$ has strong finite model property (SFMP) and intuitionistic logic $\mathsf{Int}$ can be embedded into $\mathsf{BPL}$. The technique for proving SFMP is to construct finite syntatic model in which an interpolation lemma for $\mathsf{L_{RBL}}$ is used. Consequently, it follows that the class of residuated basic algebras has the finite embeddability property (FEP), that $\mathsf{BPL}$ is PSPACE-complete, and that intuitionistic logic can be embedded into the modal logic $\mathsf{K4}$. The section 2 is devoted to recall some basic notations and remind some results for $\mathsf{RBL}$ in \cite{ML14}. In section 3 we sketch Buszkowski's proof for that the lattice order distributive residuated groupoid has FEP since we will follow the same strategy to prove the FEP of the class of residuated basic algebras. In section 5, we show that there exists a translation, a polynomial reduction from $\mathsf{Int}$ to $\mathsf{RBL}$, via which $\mathsf{Int}$ is embedded into $\mathsf{BPL}$. The structural rule free sequent calculus $\mathsf{G4ip}$ for $\mathsf{Int}$ (\cite{Dyc92,TS00}) is essentially used in our proof.

\section{Residuated Basic Logic}

We recall some definitions and results in the part I of this paper (\cite{ML14}). A residuated groupoid ($\mathbf{RG}$) is an algebra of the form $(\mathsf{G},\cdot,\leftarrow, \rightarrow,\leq)$, where $(\mathsf{G},\leq)$ is a poset and $\cdot$, $\leftarrow$ and $\rightarrow$ are a binary operations on $G$ satisfying the following conditions for all $a,b,c\in G$: 
\begin{center}
$a\cdot b\leq c$ iff $b\leq a\rightarrow c$ iff $a\leq c\leftarrow b$.
\end{center}
A {\em residuated basic algebra} ($\mathbf{RBA}$)  is an algebra $\mathbf{A}=(\mathsf{A},\wedge,\vee,\top,\bot,\rightarrow,\leftarrow,\cdot)$ such that $(\mathsf{A},\wedge,\vee,\top,\bot)$ is a bounded distributive lattice and $(\mathsf{A},\rightarrow, \leftarrow, \cdot,\leq)$ is a residuated groupoid satisfying the following axioms: for all $a,b,c\in\mathsf{A}$,
\begin{align*}
(\mathrm{w}_1)~& a\cdot\top \leq a;~~
(\mathrm{w}_2)~\top\cdot a\leq a;~~
(\mathrm{c_r})~a\cdot b\leq (a\cdot b)\cdot b
\end{align*}
where $\leq$ is the lattice order. Let $\mathbb{RBA}$ be the class of all residuated basic algebras.

Let us recall some notions of residuated basic logic $\mathsf{RBL}$. The language $\mc{L}_{\mathrm{RBL}}$ for $\mathsf{RBL}$ is the extension of $\mathsf{BPL}$ by adding binary operators $\cdot$ and $\leftarrow$. The set of all $\mc{L}_{\mathrm{RBL}}$-formulae is defined recursively as follows:
\begin{center}
$A::=p\mid\bot\mid\top\mid A\wedge A\mid A\vee A\mid A\cdot A\mid A\imp A\mid A\leftarrow A$
\end{center}
where $p\in\mathsf{Prop}$. The residuated basic logic $\mathsf{RBL}$ is the set of all $\mc{L}_{\mathrm{RBL}}$-formulae which are valid in all residuated basic algebras.

The algebraic system $\mathsf{S_{RBL}}$ for residuated basic algebras consists of the following axioms and rules:
\begin{displaymath}
(\mathrm{Id})~A \Rightarrow A\quad (\bot)~\bot\Rightarrow A\quad (\top)~A \Rightarrow \top \quad(\mathrm{Cut})~ \frac{A\Rightarrow B \quad B\Rightarrow C}{A\Rightarrow C}
\end{displaymath}
\begin{displaymath}
(\mathrm{D})~ A\wedge(B\vee C)\Rightarrow (A\wedge B)\vee (A\wedge C)
\end{displaymath}
\begin{displaymath}
(\mathrm{W}_l)~ A\cdot \top \Rightarrow A\quad (\mathrm{W}_r)~ \top\cdot A\Rightarrow A
\quad
\mathrm{(RC)}~ A\cdot B\Rightarrow (A\cdot B)\cdot B
\end{displaymath}
\begin{displaymath}
(\mathrm{R1})~ \frac{A\cdot B \Rightarrow C}{B\Rightarrow A\rightarrow C}\quad
(\mathrm{R2})~\frac{B \Rightarrow A\rightarrow C}{A\cdot B\Rightarrow C}
\end{displaymath}
\begin{displaymath}
(\mathrm{R3})~ \frac{A\cdot B \Rightarrow C}{A\Rightarrow C\leftarrow B}\quad
(\mathrm{R4}) ~\frac{A \Rightarrow C\leftarrow B}{A\cdot B\Rightarrow C}
\end{displaymath}
\begin{displaymath}
\mathrm{(\wedge L)}~\frac{A_i\Rightarrow B}{A_1\wedge A_2 \Rightarrow B},~{i\in\{1,2\}}\quad
\mathrm{(\wedge R)}~\frac{C\Rightarrow A\quad C \Rightarrow B}{C \Rightarrow A\wedge B}
\end{displaymath}
  \begin{displaymath}
(\mathrm{\vee L})~\frac{A\Rightarrow C\quad B\Rightarrow C}{A\vee B \Rightarrow C}\quad
(\mathrm{\vee R)}~ \frac{C \Rightarrow A_i}{C \Rightarrow A_1\vee A_2},~{i\in\{1,2\}}
   \end{displaymath}

The $\mc{L}_{\mathrm{RBL}}$-formula structures are defined as follows: ($\romannumeral1$) every $\mc{L}_{\mathrm{RBL}}$-formula is a formula structure; ($\romannumeral2$) if $\Gamma$ and $\Delta$ are formula structures, then $\Gamma \odot \Delta$ and $\Gamma \owedge \Delta$ are formula structures. Each formula structure $\Gamma$ is associated with a formula $\mu(\Gamma)$ defined as follows:
(i) $\mu(A) = A$ for every $\mc{L}_{\mathrm{RBL}}$-formula $A$;
(ii) $\mu(\Gamma \odot\Delta) = \mu(\Gamma)\cdot \mu(\Delta)$;
(iii) $\mu(\Gamma \owedge \Delta) = \mu(\Gamma)\wedge \mu(\Delta)$.
Sequents are of the form $\Gamma \Rightarrow A$ such that $\Gamma$ is an $\mc{L}_{\mathrm{RBL}}$-formula structure and $A$ is an $\mc{L}_{\mathrm{RBL}}$-formula.

The sequent calculus $\mathsf{L_{RBL}}$ for $\mathsf{S_{RBL}}$ consists of the following axioms and rules:
\begin{displaymath}
(\mathrm{Id}) \quad A \Rightarrow A \quad(\top)\quad A\Rightarrow \top\quad (\bot)\quad \bot\Rightarrow A
\end{displaymath}
\begin{displaymath}
(\mathrm{\imp L}) \quad\frac{\Delta \Rightarrow A;  \quad \Gamma[B] \Rightarrow C}{\Gamma[\Delta \odot (A \imp B)] \Rightarrow C} \quad(\mathrm{\imp R})\quad \frac{A \odot \Gamma \Rightarrow B}{\Gamma \Rightarrow A \imp B}
\end{displaymath}
\begin{displaymath}
(\mathrm{\leftarrow L})\quad \frac{\Gamma[A] \Rightarrow C;\quad \Delta \Rightarrow B}{\Gamma[(A\leftarrow B) \odot \Delta] \Rightarrow C} \quad (\mathrm{\leftarrow R})\quad \frac{\Gamma \odot B \Rightarrow A}{\Gamma \Rightarrow A\leftarrow B}
\end{displaymath}
\begin{displaymath}
(\mathrm{\cdot L})\quad\frac{\Gamma[A \odot B] \Rightarrow C}{\Gamma[A \cdot B] \Rightarrow C} \quad (\mathrm{\cdot R}) \quad\frac{\Gamma \Rightarrow A; \quad \Delta \Rightarrow B}{\Gamma \odot \Delta \Rightarrow A \cdot B}
\end{displaymath}
\begin{displaymath}
\mathrm{(\wedge L)}\quad\frac{\Gamma[A\owedge B]\Rightarrow C}{\Gamma[A\wedge B]\Rightarrow C}\quad\mathrm{(\wedge R)}\quad\frac{\Gamma\Rightarrow A\quad \Gamma \Rightarrow B}{\Gamma \Rightarrow A\wedge B}
\end{displaymath}
\begin{displaymath}
(\mathrm{\vee L})\quad\frac{\Gamma[A]\Rightarrow C\quad \Gamma[B]\Rightarrow C}{\Gamma[A\vee B] \Rightarrow C}\quad(\mathrm{\vee R)}\quad \frac{\Gamma \Rightarrow A_i}{\Gamma\Rightarrow A_1\vee A_2}\quad(i=1,2)
\end{displaymath}
\begin{displaymath}
(\mathrm{\owedge C}) \quad\frac{\Gamma[\Delta\owedge \Delta]\Rightarrow A}{\Gamma [\Delta]\Rightarrow A}\quad (\mathrm{\odot C}) \quad \frac{\Gamma[(\Lambda\odot\Delta)\odot \Delta]\Rightarrow A}{\Gamma[\Lambda\odot\Delta]\Rightarrow A}\quad(\Lambda~\mathrm{is~not~empty})
\end{displaymath}
\begin{displaymath}
(\mathrm{\owedge E})\quad\frac{\Gamma[\Delta\owedge \Lambda]\Rightarrow A}{\Gamma[ \Lambda\owedge \Delta]\Rightarrow A}\quad
(\mathrm{Cut}) \quad\frac{\Delta \Rightarrow A; \quad \Gamma[A] \Rightarrow B}{\Gamma[\Delta] \Rightarrow B}
\end{displaymath}
\begin{displaymath}
(\mathrm{W}^1) \quad\frac{\Gamma[\Delta]\Rightarrow A}{\Gamma[\Delta'*\Delta]\Rightarrow A}\quad(\mathrm{W}^2) \quad\frac{\Gamma[\Delta]\Rightarrow A}{\Gamma[\Delta*\Delta']\Rightarrow A}\quad(*\in\{\owedge,\odot\})
\end{displaymath}
\begin{displaymath}
(\mathrm{\owedge A^1}) \quad \frac{\Gamma[(\Delta_1\owedge \Delta_2)\owedge \Delta_3] \Rightarrow A}{\Gamma[\Delta_1\owedge (\Delta_2\owedge \Delta_3)] \Rightarrow A}\quad (\mathrm{\owedge A^2}) \quad \frac{\Gamma[\Delta_1\owedge (\Delta_2\owedge \Delta_3)] \Rightarrow A}{\Gamma[(\Delta_1\owedge \Delta_2)\owedge \Delta_3] \Rightarrow A}
\end{displaymath}

It is known \cite{ML14} that $\mathsf{L_{RBL}}$ has the cut elimination, subformula property and disjunction property. Moreover, we obtain the sequent calculus $\mathsf{DFNL}$ from $\mathsf{L_{RBL}}$ by dropping $(\top)$, $(\bot)$, $(\mathrm{W^1})$, $(\mathrm{W^2})$ and $(\odot\mathrm{C})$. We prove in \cite{ML14} that residuated basic logic is a conservative extension of Visser's basic propositional logic ($\mathsf{BPL}$) in \cite{visser81}, i.e., for any  $\mathcal{L}_{{\mathrm{BPL}}}$-formula $A$, $\vdash_{\mathsf{BPL}} A$ iff $\vdash_{\mathrm{L_{RBL}}} \Imp A$.

\begin{theorem}\label{theorem:conservative1}
For any $\mathcal{L}_{{\mathrm{BPL}}}$-formula $A$, $\vdash_{\mathrm{L_{RBL}}} \top \Imp A$ iff $\vdash_{\mathrm{BPL}} A$
\end{theorem}

\section{Algebras and Finite Syntactical Models}
 A lattice order residuated groupoid ($\mathbf{LRG}$) is an algebra $(\mathsf{G},\wedge,\vee,\cdot,\leftarrow, \rightarrow)$ such that $(\mathsf{G},\wedge,\vee)$ is a lattice and $(\mathsf{G},\cdot,\leftarrow, \rightarrow)$ is a residuated groupoid. A lattice order residuated groupoid is distributive, if its lattice reduct $(\mathsf{G},\wedge,\vee)$ is distributive. A $\mathbf{LRG}$ is called bounded, if its lattice reduct $(\mathsf{G},\wedge,\vee)$ has a greatest element $\top$ and a least element $\bot$. Both algebras are denoted by $\mathbf{DLRG}$ and $\mathbf{BLRG}$, respectively. $\mathbf{BDLRG}$ is defined naturally. Obviously, a residuated basic algebra is an $\mathbf{BDLRG}$ satisfying conditions ($\mathrm{w_1}$), ($\mathrm{w_2}$) and ($\mathrm{c_r}$).

A way of constructing a lattice order residuated groupoid by using an closure operator has been considered in literatures \cite{FA08,OT99,Bus11}. We describe this construction briefly. Let $\mathbf{G}=$($\mathsf{G}$, $\cdot$) be a groupoid. We define the following operations over the powerset $\wp(\mathsf{G})$: 
\begin{align*}
U\odot V&=\{a\cdot b \in \mathsf{G}: a\in U, b\in V\}\\
U\imp V&=\{a\in \mathsf{G} :U\odot \{a\} \subseteq V\}\\
V\leftarrow U&=\{a\in \mathsf{G}: \{a\}\odot U\subseteq V\}\\
U\vee V&=U\cup V\\
U\wedge V&=U\cap V.
\end{align*}
The powerset $\wp(\mathsf{G})$ with these operations yields a complete distributive lattice order groupoid.

An operator $C:\wp(\mathsf{G}) \rightarrow \wp(\mathsf{G})$ is called a {\em closure operator} (shortly nucleus) on $\mathbf{G}$, if it satisfies the following conditions:
\begin{enumerate}
\item[$\quad$] (C1)~~$U\subseteq C(U)$.
\item[$\quad$] (C2)~~ if $U\subseteq V$, then $C(U)\subseteq C(V)$.
\item[$\quad$] (C3)~~ $C(C(U))\subseteq C(U)$.
\item[$\quad$] (C4)~~ $C(U)\odot C(V) \subseteq C(U \odot V)$.
\end{enumerate}
For $U\subseteq G$, U is called \textit{C-closed} if $U=C(U)$. By $\mathsf{C(G)}$ we denote the family of all $C$-closed subsets of $\mathsf{G}$. Let $U\otimes V=C(U\odot V)$ and $U\vee_C V=C(U\vee V)$. It is easy to check that $\mathbf{C(G)}=(\mathsf{C(G)},\odot,\wedge,\vee_C, \imp, \leftarrow)$ is a lattice order residuated groupoid which needs not to be distributive (\cite{GO07}), where the order is $\subseteq$. 

In \cite{BusF09}, Buszkowski and Farulewski introduce an interpolation lemma to construct a finite syntactical model for $\mathrm{DFNL(\Phi)}$. We recall some definitions and notations first. Henceforth, we always assume that $\Phi$ is a finite set of simple sequents ($A\Rightarrow B$). let $\mathrm{T}$ be a set of formulae. By a $\mathrm{T}$-sequent we mean a sequent such that all formulae occurring in it belong to $\mathrm{T}$. We write $\Phi \vdash_{S} \Gamma \Rightarrow_{\mathrm{T}}A$ if $\Gamma \Rightarrow A$ has a deduction from $\Phi$ in system $S$ which consists of $\mathrm{T}$-sequents only. Two formulae $A$ and $B$ are called $\mathrm{T}$-equivalence in $S$, if $\vdash_{S}A\Leftrightarrow B$.

\begin{lemma}[\cite{BusF09}]\label{lemma:dfnl}
Let $\mathrm{T}$ be a nonempty set of all subformulae of formulae in $\Gamma\Imp A$, $\Phi$ and closed under $\wedge$ and $\vee$. If $\Phi\vdash_{\mathrm{DFNL}}\Gamma[ \Delta] \Rightarrow_{\mathrm{T}} A$, then there exists $D\in \mathrm{T}$ such that $\Phi\vdash_{\mathrm{DFNL}} \Delta \Rightarrow_{\mathrm{T}} D$ and $ \Phi\vdash_{\mathrm{DFNL}}\Gamma[D] \Rightarrow_{\mathrm{T}} A$.
\end{lemma}

Following \cite{Bus11,BusF09}, one can easily construct a finite syntactical model for any extensions of $\mathsf{DFNL}$ such that the above interpolation lemma holds. We briefly recall this construction here. Details can be found in \cite{Bus11}. Henceforth by $\overline{S}$ we mean an extension of $\mathsf{DFNL}$ satisfying Lemma \ref{lemma:dfnl}.

Let $\mathrm{T}$ be a nonempty set of formulae and closed under $\wedge$ and $\vee$. By $\mathrm{T^*}$, we denote the set of all formula structures formed out of formulae in $\mathrm{T}$. Similarly, $\mathrm{T^*[-]}$ denotes the set of all contexts in which all formulae belong to $\mathrm{T}$. $\mathbf{G}(\mathrm{T})=(\mathrm{T^*},\cdot)$ is a groupoid. Let $\Gamma[-] \in \mathrm{T^*[-]}$ and $A \in \mathrm{T}$. We define:
\begin{align*}
[\Gamma[-], A]= \{\Delta:\,\Delta\in \mathrm{T^*}~\mathrm{and}~\Phi\vdash_{\overline{S}} \Gamma[\Delta] \Rightarrow_{\mathrm{T}} A\}\\
[A]=[-, A]=\{\Gamma:\, \Gamma\in \mathrm{T^*}~\mathrm{and}~\Phi\vdash_{\overline{S}}\Gamma \Rightarrow_{\mathrm{T}} A\}
\end{align*}
Let $B(\mathrm{T})$ be the family of all sets $[\Gamma[-], A]$ defined above. Define $C_{\mathrm{T}}$ by:
\begin{align*}
C_{\mathrm{T}}(U)=\bigcap\{[\Gamma[-], A] \in B(\mathrm{T}): U\subseteq [\Gamma[-], A]\}
\end{align*}
It can be shown that $C_{\mathrm{T}}$ satisfies (C1)-(C4), and so $C_{\mathrm{T}}$ is an closure operator (\cite{BusF09}). The algebra $\mathbf{C_T}(\mathbf{G}(\mathrm{T^*}))$ satisfies all the laws defining lattice order residuated groupoid, but needs not to be distributive. The following equations are true in $\mathbf{C_T}(\mathbf{G}(\mathrm{T^*}))$ provided that all formulae appearing in them belong to $\mathrm{T}$ (\cite{BusF09}):
\begin{equation}
[A] \otimes [B]=[A \cdot B], \quad [A]\imp [B]=[A\imp B],\quad [A]\leftarrow[B]=[A\leftarrow B]
\end{equation}
\begin{equation}
[A]\cap[B]=[A\wedge B], \quad [A]\vee_C[B]=[A\vee B]
\end{equation}
Since $\mathrm{T}$ is closed under $\wedge$ and $\vee$, by Lemma 2.1 and equations (I) and (II), the algebra $\mathbf{C_T}(\mathbf{G}(\mathrm{T^*}))$ is a $\mathbf{BLRG}$. In fact one can prove that for any $U\in \mathsf{C_T(\mathsf{T^*})}$, there exists a formula $A\in \mathrm{T}$ such that $U=[A]$. Obviously $\mathrm{T}$ is finite up to the relation of $T$-equivalence in $\overline{S}$. Hence there are only finitely many sets $[A]$. Then $C_{\mathrm{T}}(\mathrm{T^*})$ is finite. By Lemma 2.1 and the distributive law, the following inequation holds in $C_{\mathrm{T}}(\mathrm{T^*})$:
\begin{equation}
U\wedge(V\vee_c W)\subseteq (U\wedge V)\vee_c (U\wedge W)
\end{equation}

\begin{theorem}
The algebra $\mathbf{C_T}(\mathbf{G}(\mathrm{T^*}))$ is finite and belongs to $\mathbb{DLRG}$.
\end{theorem}

Let $\bot_C=C(\emptyset)$ and $\top=\mathsf{G}$. Then the algebra $\mathbf{C_T}(\mathbf{G}(\mathrm{T^*}))$ is a finite $\mathbf{BDLRG}$.

\section{Interpolation and FMP}
By the FMP of $\mathsf{L_{RBL}}$ we mean that any sequent $\Gamma\Rightarrow A$ not provable in $\mathsf{L_{RBL}}$ is refutable in a residuated basic algebra. The algebraic completeness of $\mathsf{L_{RBL}}$ w.r.t $\mathbb{RBA}$ follows from FMP immediately. By the SFMP of $\mathrm{L_{RBL}}$ we mean that for any sequent $\Gamma\Rightarrow A$ not derivable from $\Phi$ in $\mathrm{L_{RBL}}$ there exists  a residuated basic algebra $\mathbf{A}$ such that all sequents in $\Phi$ are valid in $\mathbf{A}$ but $\Gamma\Rightarrow A$ is not.

A model for $\mathsf{L_{RBL}}$ is a pair $(\mathbf{G}, \sigma)$ such that $\mathbf{G}\in \mathbb{RBA}$ and $\sigma$ is an valuation in $\mathbf{G}$. Each valuation $\sigma$ is extended for formulae and formula structures as follows:
\begin{displaymath}
\sigma(A\cdot B)=\sigma(A)\cdot \sigma(B),\quad \sigma(\top)=\top, \quad\sigma(\bot)=\bot
\end{displaymath}
\begin{displaymath}
\sigma(A\imp B)=\sigma(A)\imp \sigma(B),\quad \sigma(A \leftarrow B)=\sigma(A)\leftarrow\sigma(B)
\end{displaymath}
\begin{displaymath}
 \sigma(A\wedge B)=\sigma(A)\wedge\sigma(B),\quad \sigma(A\vee B)=\sigma(A)\vee\sigma(B)
\end{displaymath}
\begin{displaymath}
\sigma(\Gamma \odot \Delta)=\sigma(\Gamma)\cdot\sigma(\Delta),\quad \sigma(\Gamma \owedge \Delta)=\sigma(\Gamma)\wedge \sigma(\Delta)
\end{displaymath}
A sequent $\Gamma \Rightarrow A$ is true in model $(\mathbf{G}, \sigma)$, if $\sigma(\Gamma)\leq \sigma(A)$ in $\mathbf{G}$. 

We prove the interpolation lemma for $\mathsf{L_{RBL}}$ and employ the proof technique described in section 3 to show the SFMP for $\mathrm{L_{RBL}}$. Let $\mathrm{T}$ be a set of $\mathcal{L}_{\mathrm{RBA}}$-formulae containing $\bot$ and $\top$ and closed under $\wedge$ and $\vee$.

\begin{lemma}\label{lemma:int}
If $\Phi \vdash_{\mathsf{L_{RBL}}}\Gamma[ \Delta] \Rightarrow_{\mathrm{T}} A$, then there exists $D\in \mathrm{T}$ such that $\Phi\vdash_{\mathsf{L_{RBL}}} \Delta \Rightarrow_{\mathrm{T}} D$ and $\Phi\vdash_{\mathsf{L_{RBL}}}\Gamma[D] \Rightarrow_{\mathrm{T}} A$.
\end{lemma}
\begin{proof}
 If $\vdash_{\mathrm{L_{RBL}}} \Gamma[\Delta]\Rightarrow A$ and formula $D$ satisfying the properties given statement of lemma, then we call $D$ an interpolant of $ \Delta$. 

The proof proceeds by induction on $\mathrm{T}$-derivation of $\Gamma[ \Delta] \Rightarrow A$. The case of axioms are easy. For $A \Rightarrow A$, $A\Rightarrow \top$ and $\bot\Rightarrow A$, we have $\Delta=A$ or $\Delta=\bot$. Hence $A$ and $\bot$ are the interplants of $\Delta$, respectively.

Let $\Gamma[\Delta ]\Rightarrow A$ be the conclusion of the rule $\mathrm{\mathrm{R}}$. For the case $\mathrm{R=(Cut)}$, it is easy. If $\Delta$ comes from one premise of (Cut), then one takes an interpolant from this premise. Otherwise, $\Delta$ comes from $\Delta'[C]$ in a premise where $C$ is the cut formula. Then an interpolant of $\Delta'[C]$ is also one of $\Delta$. Let us consider other rules.

(1) Assume that $\Delta$ contains no formula or structure operation introduced by $\mathrm{R}$ (no active formula or structure operation). Consider the following subcases.

(1.1) $\mathrm{R}=\mathrm{(\wedge R)}$. Assume that the premises are $\Gamma[ \Delta] \Rightarrow A_1$ and $\Gamma[ \Delta ] \Rightarrow A_2$, and the conclusion is $\Gamma[ \Delta] \Rightarrow A_1\wedge A_2$. By induction hypothesis, there are interpolants $D_1$, $D_2$ such that $\Phi\vdash_{\mathrm{L_{RBL}}} \Delta  \Rightarrow_{\mathrm{T}} D_1$,  $\vdash_{\mathrm{L_{RBL}}}\Gamma[D_1] \Rightarrow_{\mathrm{T}} A_1$, $\vdash_{\mathrm{L_{RBL}}} \Delta  \Rightarrow_{\mathrm{T}} D_2$  and $\vdash_{\mathrm{L_{RBL}}} \Gamma[D_2]\Rightarrow_{\mathrm{T}} A_2$.  Then one gets $\vdash_{\mathrm{L_{RBL}}} \Delta  \Rightarrow_{\mathrm{T}} D_1\wedge D_2$ by $\mathrm{(\wedge R)}$. By (W), ($\mathrm{\wedge L}$) and ($\mathrm{\wedge R}$), one obtains $\vdash_{\mathrm{L_{RBL}}} \Gamma[D_1\wedge D_2] \Rightarrow_{\mathrm{T}} A_1\wedge A_2$.

(1.2) $\mathrm{R}=(\mathrm{\vee L})$. Assume that the premises are $\Gamma[B][\Delta] \Rightarrow A$ and $\Gamma[C][\Delta] \Rightarrow A$, and the conclusion is $\Gamma[B\vee C][\Delta] \Rightarrow A$. By induction hypothesis, there are interpolants $D_1$, $D_2$ of $\Delta$ in the premises. Then $D_1\vee D_2$ is an interpolant of $\Delta$ by (W), ($\mathrm{\wedge L}$), ($\mathrm{\vee L}$) and ($\mathrm{\wedge R}$).

(1.3) $\mathrm{R}=\mathrm{(\owedge C)}$. Assume that the premise is $\Gamma'[\Delta' \owedge \Delta'] \Rightarrow A$ and the conclusion is $\Gamma'[ \Delta']\Rightarrow A$. If $\Delta'$ is contained in $\Delta$ including the case $\Delta=\Delta'$, then by induction hypothesis, the interpolant $D$ of the source of $\Delta$ in the premise is also an interpolant of $\Delta$ in the conclusion. Otherwise, assume $\Delta'=\Delta''[\Delta]$. By inductive hypothesis, there exist $D_1, D_2\in \mathrm{T}$ such that $\vdash_{\mathrm{L_{RBL}}} \Delta  \Rightarrow_{\mathrm{T}} D_1$, $\vdash_{\mathrm{L_{RBL}}} \Delta  \Rightarrow_{\mathrm{T}} D_2$ and $\vdash_{\mathrm{L_{RBL}}}\Gamma'[\Delta''[D_1]\owedge \Delta''[D_2]] \Rightarrow_T A$. By ($\mathrm{\wedge R}$), one gets $\vdash_{\mathrm{L_{RBL}}} \Delta  \Rightarrow_{\mathrm{T}} D_1\wedge D_2$.
By (W) and ($\mathrm{\wedge L}$), one obtains $\vdash_{\mathrm{L_{RBL}}}\Gamma'[\Delta''[D_1\wedge D_2]\owedge \Delta''[D_1\wedge D_2]] \Rightarrow_T A$. Hence by ($\owedge C$), $\vdash_{\mathrm{L_{RBL}}}\Gamma'[\Delta''[D_1\wedge D_2]] \Rightarrow_{\mathrm{T}} A$. Hence $D_1\wedge D_2 \in \mathrm{T}$  is an interpolant of $\Delta$.

(1.4) $\mathrm{R}=\mathrm{(\odot C)}$. The proof is quite similar to the case $\mathrm{R}=\mathrm{(\owedge C)}$.

(1.5) $\mathrm{R}=\mathrm{(\owedge E)}$, $\mathrm{(\owedge A^1)}$ or $\mathrm{(\owedge A^2)}$. The proof is quite similar to the first subcase of the case $\mathrm{R}=\mathrm{(\owedge C)}$,

(1.6) For the other cases, $\Delta$ must come from exactly one premise of R. Then an interpolant of $\Delta$ in this premise is also an interpolant of $\Delta$ in the conclusion.

(2) Assume that $ \Delta $ contains active formula or structure operation. If $\Delta$ is a single formula $E$, then $E$ is an interpolant of $\Delta$. Otherwise, let us consider the following subcases.

(2.1) $\mathrm{R=(\backslash L)}$ or $\mathrm{R=(/ L)}$. Let $\mathrm{R=(\backslash L)}$. Assume that the premises are $\Gamma'[C]\Rightarrow A$ and $\Delta'\Rightarrow B$, and the conclusion is $\Gamma'[\Delta'\circ B\backslash C]\Rightarrow A$. Then, $\Delta $ contains $\Delta'\circ B\backslash C$. Assume that $\Delta''[C]$ occurs in $\Gamma'[C]$, and $\Delta=\Delta''[\Delta'\circ B\backslash C]$. Then an interpolant $D$ of $\Delta''[C]$ is also an interpolant of $ \Delta $. For ($\mathrm{/L}$), the arguments is similar.

(2.2) $R=\mathrm{(\vee L)}$. Assume that $\Delta =\Delta'[B_1\vee B_2]$, the premises are $\Gamma[ \Delta'[B_1] ]$ $ \Rightarrow A$ and $\Gamma[ \Delta'[B_2] ] \Rightarrow A$, and the conclusion is $\Gamma[\Delta'[B_1\vee B_2]]\Rightarrow A$. Let $D_1$ be an interpolant of $ \Delta'[B_1]$ in the first premise and $D_2$ be an interpolant of $ \Delta'[B_2]$ in the second premise. Hence $D_1\vee D_2$ is an interpolant of $ \Delta $ in the conclusion by  ($\mathrm{\vee R}$) and ($\mathrm{\vee L}$).

(2.3) $\mathrm{R=(\wedge L)}$ or $\mathrm{R=(\cdot L)}$. Let $\mathrm{R=(\wedge L)}$. Assume that $\Delta=\Delta'[B\wedge C]$, the premise is $\Gamma'[B\owedge C]\Rightarrow A$, and the conclusion is $\Gamma'[B\wedge C]\Rightarrow A$. Then $\Delta'[B\owedge C]$ occurs in $\Gamma'[B\owedge C]$. Hence the interpolant $D$ of $\Delta'[B\owedge C]$ is also an interpolant of $\Delta$ in the conclusion. The arguments for ($\mathrm{\cdot L}$) is similar.

(2.4) $\mathrm{R=\mathrm{(W^1)}}$. Assume that the premise of is $\Gamma'[\Upsilon]\Rightarrow A$ and the conclusion is $\Gamma'[\Upsilon*\Delta']\Rightarrow A$. If $\Delta=\Delta'$ or $\Delta$ is contained in $\Delta'$ then $D=\top$ is an interpolant of $\Delta$ in the conclusion. Otherwise, assume that $\Delta$ is obtained from $\Delta''$. By induction hypothesis, the interpolant of $\Delta''$ is also an interpolant of $\Delta$ in the conclusion.
\end{proof}

Let $\mathrm{T}$ be a set of $\mathcal{L}_{\mathrm{RBA}}$-formulae containing $\bot$ and $\top$ and closed under $\wedge$ and $\vee$. $\mathbf{C_T}(\mathbf{G}(\mathrm{T^*}))$ is defined as above. Consequently, $\mathbf{C_T}(\mathbf{G}(\mathrm{T^*}))$ is a finite $\mathbb{BDLRG}$. Further we show that the following inequations hold in $\mathbf{C_T}(\mathbf{G}(\mathrm{T^*}))$:
\begin{displaymath}
U\otimes V\subseteq U,\quad U\otimes V\subseteq V,\quad
U\otimes (V\otimes V)\subseteq U \otimes V.
\end{displaymath}
It suffices to show that $[A]\otimes[B]\subseteq [A]$, $[A]\otimes [B] \subseteq [B]$ and $[A]\otimes [B]\subseteq ([A] \otimes [B])\otimes [B]$. By equations (II) and (I), it suffices to show that $[A\cdot B]\subseteq [A]$, $[A\cdot B]\subseteq [B]$ and $[A\cdot B]\subseteq [(A\cdot B)\cdot B]$. Obviously, since $A\cdot B\Rightarrow A$ $A\cdot B \Rightarrow B$ and $A\cdot B\Rightarrow (A\cdot B)\cdot B$ are axioms in $\mathrm{L_{RBL}}$, these inequations hold in $\mathbf{C_T}(\mathbf{G}(\mathrm{T^*}))$.
Hence we obtain the following theorem.

\begin{theorem}
The algebra $\mathbf{C_T}(\mathbf{G}(\mathrm{T^*}))$ is a finite residuated basic algebra.
\end{theorem}
\begin{lemma}
Assume $\Phi\not\vdash_{\mathsf{L_{RBL}}}\Gamma \Rightarrow A$. There exist a finite $\mathbf{G} \in \mathbb{RBA}$ and a valuation $\sigma$ such that all sequents in $\Phi$ are true in ($\mathbf{G}$, $\sigma$) but $\Gamma \Rightarrow A$ is not.
\end{lemma}
\begin{proof}
Suppose that $\mathrm{T}$ is the set of all formulae appearing in $\Gamma \Rightarrow A$, containing $\bot$, $\top$ and closed under $\wedge$ and $\vee$. Let $\mathbf{G}=\mathbf{C_T}(\mathbf{G}(\mathsf{T^*}))$ and $\sigma(p)=[p]$ for $p\in T$. By (I)-(II), we get $[A]=\sigma(A)$, for $A\in T$. Assume that $\Gamma \Rightarrow A$ is true in ($\mathbf{C_T}(\mathbf{G}(\mathrm{T^*}))$, $\sigma$). Then $\sigma(\Gamma)\subseteq\sigma(A)$. Since $\Gamma\in \sigma(\Gamma)$, we get $\Gamma\in\sigma(A)=[A]$. Hence $\vdash_{\mathrm{L_{RBL}}} \Gamma \Rightarrow_{T} A$, which yields a contradiction.
\end{proof}

\begin{theorem}
$\mathsf{L_{RBL}}$ has SFMP.
\end{theorem}
\begin{theorem}
The logic $\mathsf{RBL}$ is decidable.
\end{theorem}
If a class of algebras $\mb{K}$ is closed under (finite) products, then SFMP for $\mb{K}$ is equivalent to FEP for $\mb{K}$, i.e., every finite
partial subalgebra of an algebra from $\mb{K}$ is embeddable into a finite algebra from $\mb{K}$ (\cite{GO07}). Then it follows immediately that $\mathbb{RBA}$ has FEP.

\section{Embedding of $\mathsf{Int}$ into $\mathsf{BPL}$}
An $\mc{L}_{\mathrm{Int}}$-formula $A$ is built from propositional letters and $\bot$ using $\wedge,\vee$ and the intuitionistic implication $\imp$.
An $\mc{L}_{\mathrm{Int}}$-formula structure, which is a finite (possibly empty) sequence of formulae (in fact, the order of formulae do not matter), is defined as follows: (i) each Int-formula is a Int-formula structure; (ii) if $\Gamma$ and $\Delta$ are Int-formula structures, then $(\Gamma,\Delta)$ is a Int-formula structure. An $\mc{L}_{\mathrm{Int}}$-sequent is of the form $\Gamma\Imp A$ where $\Gamma$ is a $\mc{L}_{\mathrm{Int}}$-formula structure and $A$ is an $\mc{L}_{\mathrm{Int}}$-formula. The sequent calculus $\mathsf{G4ip}$ for intuitionistic logic can be found in \cite{TS00}:
\begin{displaymath}
(\mathrm{Id}) ~ p,\Gamma\Rightarrow p~(p~\mathrm{is~atomic})\quad(\bot)~\bot,\Gamma\Imp A
\end{displaymath}
\begin{displaymath}
(\mathrm{\wedge L})~ \frac{A, B,\Gamma \Rightarrow C}{A\wedge B,\Gamma C} \quad
(\mathrm{\wedge R})~ \frac{\Gamma \Imp A\quad \Delta\Imp B}{\Gamma,\Delta \Rightarrow A\wedge B}\quad
(\mathrm{\imp R})~\frac{A, \Gamma \Rightarrow B}{\Gamma \Rightarrow A \imp B}
\end{displaymath}
\begin{displaymath}
(\mathrm{\vee L})~\frac{A, \Gamma\Rightarrow C\quad B, \Delta\Rightarrow C}{A\vee B, \Gamma,\Delta \Rightarrow C}\quad(\mathrm{\vee R)}~ \frac{\Gamma \Rightarrow A_i}{\Gamma\Rightarrow A_1\vee A_2}\quad(i=1,2)
\end{displaymath}
\begin{displaymath}
(\mathrm{\imp L_1})~\frac{p, B, \Gamma \Rightarrow E}{p\imp B,p,\Gamma \Rightarrow E}~(p~\mathrm{is~atomic})\quad
(\mathrm{\imp L_2})~\frac{C\imp (D\imp B),\Gamma \Rightarrow E}{C\wedge D\imp B,\Gamma \Rightarrow E}
\end{displaymath}
\begin{displaymath}
(\mathrm{\imp L_3})~\frac{C\imp B, D\imp B,\Gamma\Imp E}{C\vee D\imp B,\Gamma\Rightarrow E} \quad
(\mathrm{\imp L_4})~\frac{D\imp B, C, \Gamma\Imp D\quad B,\Gamma\Imp E}{(C\imp D)\imp B,\Gamma \Rightarrow E}
\end{displaymath}
\begin{definition}[\cite{TS00}]
The {\em weight} of an $\mc{L}_{\mathrm{Int}}$-formula $A$ is a natural number defined recursively as follows:
\begin{itemize}
\item $w(p) = w(\bot) = 2$ for each propositional letter $p$.
\item $w(A\wedge B) = w(A)(1+w(B))$.
\item $w(A\vee B) = 1+w(A)+w(B)$.
\item $w(A\imp B)=1+w(A)w(B)$.
\end{itemize}
For each $\mc{L}_{\mathrm{Int}}$-sequent $\Gamma\Imp A$, we put 
\begin{center}
$w(\Gamma\Imp A) = \sum\{w(B):B\in \Gamma$ or $B=A\}$.
\end{center}
\end{definition}
Observe that for each rule of $\mathsf{G4ip}$, the weight of each premises is lower than that of the conclusion.
This fact is used in our proof of the embedding theorem. Now let us turn the notion of positive (negative) Int-formula in an $\mc{L}_{\mathrm{Int}}$-sequent.

\begin{definition}
The {\em positiveness} ({\em negativeness}) of an Int-formula $A$ appeared in a $\mc{L}_{\mathrm{Int}}$-sequent $\Gamma\Imp C$ is defined recursively by the following rules:
\begin{itemize}
\item $A=C$ is positive, and $A\in \Gamma$ is negative.
\item if $A=A_1\wedge A_2$ is positive (negative), then both $A_1$ and $A_2$ are positive (negative).
\item  if $A=A_1\vee A_2$ is positive (negative), then both $A_1$ and $A_2$ are positive (negative).
\item  if $A=A_1\imp A_2$ is positive (negative), then $A_1$ is negative (positive) and $A_2$ is positive (negative).
\end{itemize}
\end{definition}

\begin{example}
By $v(A) = +$ and $v(A)=-$ we mean that the formula $A$ in a sequent is positive and negative respectively. Consider the sequent $A,B\imp C\Imp E\imp F$. Then $v(A)=-$, $v(B\imp C)=-$ and so $v(B)=+$ and $v(C)=-$. In the consequent, $v(E\imp F)=+$ and so $v(E)=-$ and $v(F)=+$.
\end{example}

The positiveness or negativeness of any subformula in a sequent can be calculated. For any derivation, the positiveness or negativeness of each subformula cannot be changed by applications of rules.

For any $\mc{L}_{\mathrm{Int}}$-formula $A$ and $n>0$, let $A^{\# n}$ be the formula obtained from $A$ by replacing all occurrences of its positive subformula $B$ by $\top^n\imp B$, where $\top^n\imp B$ is defined by induction on $n > 0$ as follows: $\top^1\imp B := \top \imp B$ and $\top^{n+1}\imp B:=\top\imp(\top^n\imp B)$.

\begin{example}
 Let $A=A_1\wedge A_2$ and $A$ is positive. $A^{\#n} =(A_1\wedge A_2)^{\#n}= \top^n\imp (A_1^{\#n} \wedge A_2^{\#n})$. If $A$ is negative then $A^{\#n} = A_1^{\#n} \wedge A_2^{\#n}$. Let $A=A_1\imp (A_2\imp A_3)$ and $A$ is positive. Then $A^{\#n} = \top^n\imp (A_1^{\#n}\imp(\top^n\imp( A_2^{\#n}\imp A_3^{\#n}))$.
\end{example}

\begin{definition}
We define a map $(.)^{\#n}$ from $\mc{L}_{\mathrm{Int}}$-formula structures to $\mc{L}_{\mathrm{RBL}}$-formula structures as follows:
\begin{align*}
A^{\#n} &= A^{\#n}\\
(\Gamma,\Delta)^{\#n} &= \Gamma^{\#n}\owedge \Delta^{\#n}
\end{align*}
For each $\mc{L}_{\mathrm{Int}}$-sequent $\Gamma\Imp A$, we define
\[
(\Gamma\Imp A)^\#=
\begin{cases}
\Gamma^{\#w(\Gamma\Imp A)}\Imp A^{\#w(\Gamma\Imp A)},~\mathrm{if}~\Gamma~\mathrm{is~nonempty}.\\
\top\Imp A^{\#w(\Imp A)}, ~\mathrm{otherwise}.
\end{cases}
\]
We define the translation $Tr(.):\mc{L}_{\mathrm{Int}}\imp\mc{L}_{\mathrm{BPL}}$ by putting:
\begin{center}
$Tr(A) = $ the succedent of $(\Imp A)^{\#w(\Imp A)}$.
\end{center}
\end{definition}
\begin{proposition}\label{prop:1}
The following $\mc{L}_{\mathrm{RBL}}$-sequents  are derivable in $\mathsf{L_{RBL}}$:
\begin{itemize}
\item[(1)] $A\cdot (B\wedge C)\Imp (A\cdot B)\cdot C$
\item[(2)] $A\cdot (B\cdot C)\Imp(A\cdot B )\cdot C$
\item[(3)] $A\cdot(B\wedge C)\Imp A\cdot B\wedge A\cdot C$.
\item[(4)] $(B\wedge C)\cdot A\Imp B\cdot A\wedge C\cdot A$.
\item[(5)] $A\cdot B\Imp A\wedge B$.
\item[(6)] $(B\vee C)\imp A\Leftrightarrow (B\imp A)\wedge (C\imp A)$.
\item[(7)] $A\imp (B\wedge C)\Leftrightarrow (A\imp B)\wedge (A\imp C)$.
\item[(8)] $A\cdot(B\vee C)\Leftrightarrow (A\cdot B) \vee (A\cdot C)$.
\item[(9)] $(\top^n\imp (C\wedge A))\imp B\Imp C\imp (A\imp B)$
\item[(10)] $(\top^n\imp(C\vee A))\imp B\Imp (C\imp B)\wedge (A\imp B)$
\item[(11)] $(\top^n\imp(C\imp A))\imp B\Imp (A\imp B)\wedge((C\imp A)\imp B)$
\end{itemize}
\end{proposition}
\begin{proof}
The items (3)-(8) are checked regularly. We check only (1), (2), (9), (10) and (11). Let us consider (1). From $A\Imp A$ and $B\wedge C\Imp B$, by  $(\mathrm{\cdot R)}$, we get $A\cdot (B\wedge C)\Imp A\cdot B$. Then apply $(\mathrm{\cdot R)}$ to the resulting sequent and $B\wedge C\Imp C$, we get $(A\cdot(B\wedge C))\cdot (B\wedge C)\Imp (A\cdot B)\cdot C$. Since $A\cdot(B\wedge C)\Imp (A\cdot(B\wedge C))\cdot (B\wedge C)$ is an instance of axiom, by (Cut), we get $A\cdot(B\wedge C)\Imp (A\cdot B)\cdot C$.

Let us consider (2). By ($\mathrm{W^1}$), we obtain $B\cdot C \Imp B$ and $B\cdot C\Imp C$. By ($\mathrm{\wedge R}$), we get $B\cdot C\Imp B\wedge C$. By applying ($\mathrm{\cdot R}$) to the resulting sequent and $A\Imp A$, we obtain $A\cdot(B\cdot C)\Imp A\cdot (B\wedge C)$. By (1) and (Cut), we get $A\cdot(B\cdot C)\Imp (A\cdot B)\cdot C$.

Let us consider (8). By (Id), ($\mathrm{W^1}$) and ($\imp$R), we get $A\Imp \top^n\imp A$ and $C\Imp \top^n\imp C$. Then by ($\mathrm{W^1}$), ($\mathrm{W^2}$) and ($\mathrm{\wedge R}$), we obtain $A\odot C\Imp (\top^n\imp C)\wedge (\top^n\imp A)$. By applying ($\mathrm{\imp L}$) to the resulting sequent and $B\Imp B$, we get $(A\odot C)\odot(((\top^n\imp C)\wedge (\top^n\imp A))\imp B)\Imp B$. By (2) and (Cut), we obtain $A\odot (C\odot(((\top^n\imp C)\wedge (\top^n\imp A))\imp B))\Imp B$. By ($\mathrm{\imp R}$), we get $((\top^n\imp C)\wedge (\top^n\imp A))\imp B \Imp C\imp(A\imp B)$. By (7), we obtain $((\top^n\imp C)\wedge (\top^n\imp A))\Leftrightarrow \top^n\imp(C\wedge A)$. Hence by ($\mathrm{\imp L}$) and ($\mathrm{\imp R}$), $(\top^n\imp(C\wedge A))\imp B \Imp ((\top^n\imp C)\wedge (\top^n\imp A))\imp B $. By apply (Cut) to this sequent and $((\top^n\imp C)\wedge (\top^n\imp A))\imp B \Imp C\imp(A\imp B)$, we get $(\top^n\imp(C\wedge A))\imp B\Imp  C\imp(A\imp B)$.

Let us consider (9). By (Id), ($\mathrm{W^1}$) and ($\mathrm{\imp R}$), we obtain $C\vee A \Imp \top^n \imp (C\vee A)$. By apply $\mathrm{(\imp L)}$ to this sequent and $B\imp B$, we get $((C\vee A)\odot ((\top^n\imp (C\vee A))\imp B)\Imp B$. By ($\mathrm{\imp R}$), we obtain  $(\top^n\imp (C\vee A))\imp B\Imp( C\vee A)\imp B$. By (5), we have $( C\vee A)\imp B\Leftrightarrow (C\imp B)\wedge (A\imp B)$. By (Cut), we get $(\top^n\imp (C\vee A))\imp B\Imp (C\imp B)\wedge (A\imp B)$.

Let us consider (10). By ($\mathrm{W^1}$) and ($\mathrm{\imp R}$), we get $A\Imp \top^n\imp(C\imp A)$. By applying ($\mathrm{\imp L}$) to this sequent and $B\imp B$, we get $A\odot (\top^n\imp(C\imp A))\imp B\Imp B$. Hence by ($\imp R$), we obtain $(\top^n\imp(C\imp A))\imp B\Imp A\imp B$. By similar argument, we get $(\top^n\imp(C\imp A))\imp B\Imp (C\imp A)\imp B$. Hence by ($\mathrm{\wedge R}$), we obtain $(\top^n\imp(C\imp A))\imp B\Imp  (A\imp B) \wedge((C\imp A)\imp B)$.
\end{proof}

Let $\mathsf{L_{RBL}'}$ be the sequent calculus obtained from $\mathsf{L_{RBL}}$ by replacing the axiom $(\mathrm{Id})~A\Imp A$
by the axiom $(\mathrm{Id}')~p\Imp p$ ($p$ is atomic).

\begin{lemma}\label{lemma:rblequiv}
For any $\mc{L}_{\mathrm{RBL}}$-sequent $\Gamma\Imp A$, $\vdash_{\mathsf{L_{RBL}}}\Gamma\Imp A$ iff $\vdash_{\mathsf{L_{RBL}'}}\Gamma\Imp A$.
\end{lemma}
\begin{proof}
The right-to-left direction is obvious. For the other direction it suffices to show that $(\mathrm{Id})$ is admissible in $\mathsf{L_{RBL}'}$. We proceed by induction on the complexity of $A$. The cases of $\wedge$, $\vee$ and $\cdot$ are done easily by inductive hypothesis. For $A=A_1\leftarrow A_2$, by inductive hypothesis, $\vdash_{\mathsf{L_{RBL}'}}A_1\Imp A_1$ and $\vdash_{\mathsf{L_{RBL}'}}A_2\Imp A_2$. Then by $(\leftarrow\mathrm{L})$ we get $(A_1\leftarrow A_2)\odot A_1\Imp A_2$. By $(\leftarrow\mathrm{R})$, we get $A_1\leftarrow A_2\Imp A_1\leftarrow A_2$. The case of $\imp$ is similar to the case $\leftarrow$.
\end{proof}

It follows immediately that all sequents in proposition \ref{prop:1} hold in the sequent calculus $\mathsf{L_{RBL}'}$.

For any $\mc{L}_{\mathrm{RBL}}$-sequent $\Gamma\Imp A$ and an occurrence of positive subformula $B$ in it, we define $\Gamma\Imp A[B/\top \imp B]$ as the sequent obtained from $\Gamma\Imp A$ by replacing this occurrence of $B$ by $\top\imp B$.

\begin{lemma}\label{lemma:2}
For any $\mc{L}_{\mathrm{RBL}}$-sequent $\Gamma\Imp A$ and an occurrence of positive subformula $B$ in it, if $\vdash_{\mathsf{L_{RBL}'}}\Gamma\Imp A$, then $\vdash_{\mathsf{L_{RBL}'}}\Gamma\Imp A[B/\top\imp B]$.
\end{lemma}
\begin{proof}
By induction on the derivation of $\Gamma\Imp A$ in $\mathsf{L_{RBL}'}$.

($\mathrm{Id'}$) we have $p\Imp p[B/\top\imp B] = p\Imp \top\imp p$ which is deribale in $\mathsf{L_{RBL}'}$.

($\top$) Let $A\Imp\top$. If $B=\top$, then $A\Imp \top[B/\top\imp B] = A\Imp \top\imp\top$ is derivable in $\mathsf{L_{RBL}'}$. Otherwise, $B$ is in $A$ and $A[B/\top\imp B] \Imp \top$ is an instance of axiom.

($\bot$) $B$ must be contained in $A$ and the sequent $\bot\Imp A[B/\top\imp B]$ is an instance of axiom in $\mathsf{L_{RBL}'}$.

($\rightarrow\mathrm{L}$) Let the premises be $\Delta\Imp A$ and $\Gamma[D]\Imp C$, and the conclusion $\Gamma[\Delta\odot(A\imp D)]\Imp C$. Consider the sequent $\Gamma[\Delta\odot(A\imp D)]\Imp C[B/\top\imp B]$.
Since $A\imp D$ is not positive, $B$ is in $\Delta\Imp A$ or $\Gamma[D]\Imp C$. Hence by inductive hypothesis and $(\rightarrow\mathrm{L})$, we get the required sequent. The proof of cases ($\cdot\mathrm{L}$), ($\leftarrow\mathrm{L}$) $(\wedge\mathrm{L})$, $(\vee\mathrm{L})$, $(\owedge\mathrm{C})$, $(\odot\mathrm{C})$, $(\owedge\mathrm{E})$, $(\owedge\mathrm{A^1})$, $(\owedge\mathrm{A^2})$, $(\mathrm{Cut})$, are quite similar, since none of these rules create a new positive formula in the derivation.

($\imp\mathrm{R}$) Let the premise be $A\odot\Gamma\Imp D$ and the conclusion $\Gamma\Imp A\imp D$. If $B$ is in $A\odot \Gamma\Imp D$, then by inductive hypothesis and $(\imp\mathrm{R})$, we have $\Gamma\Imp A\imp D[B/\top\imp B]$. Otherwise $B=A\imp D$. Then from $A\odot\Gamma\Imp D$. by ($\imp\mathrm{R}$) we get $\Gamma\Imp A\imp D$. Then by $(\mathrm{W^1})$ we get $\top\odot\Gamma\Imp A\imp D$. Hence by $\imp\mathrm{R}$ we get $\Gamma\Imp \top\imp(A\imp D)$.
The proof of cases ($\leftarrow\mathrm{R}$), $(\cdot\mathrm{R})$, $(\wedge\mathrm{R})$ and $(\vee\mathrm{R})$ are quite similar.

$(\mathrm{W^1})$ Let the premise be $\Gamma[ \Delta_2]\Imp C$ and the conclusion $\Gamma[\Delta_1* \Delta_2]\Imp C$. Then $B$ is in the premise. Hence by inductive hypothesis and $(\mathrm{W^1})$, we get the required sequent. Otherwise $B$ occurs in $\Delta_1$. Then by $(\mathrm{W^1})$ we get the required sequent directly. The proof of cases $(\mathrm{W^2})$ is quite similar.
\end{proof}

\begin{corollary}\label{corollary:cor1}
For any $\mc{L}_{\mathrm{RBL}}$-sequent $\Gamma\Imp A$ and $0<i<j$, if $\vdash_{\mathsf{L_{RBL}'}}\Gamma\Imp A$, then $\vdash_{\mathsf{L_{RBL}'}}\Gamma\Imp A[B^{\#i}/B^{\#j}]$.
\end{corollary}

\begin{theorem}\label{thm:g4ip}
For any  $\mc{L}_{\mathrm{Int}}$-sequent $\Gamma\Imp A$, if $\vdash_{\mathsf{G4ip}}\Gamma\Imp A$, then $\vdash_{\mathsf{L_{RBL}'}}\Gamma\Imp A)^\#$.
\end{theorem}
\begin{proof}
We proceed by induction on the derivation of $\Gamma\Imp A$ in $\mathsf{G4ip}$. It suffices to show that all rules of G4ip are admissible under the translation $\#$. The axioms $(\mathrm{Id})$ and $\mathrm{(\bot)}$ are easy. For $(\wedge\mathrm{L})$, let the premise be $A,B,\Gamma\Imp C$ with weight $i$, and the conclusion $A\wedge B, \Gamma\Imp C$ with weight $j$. Assume $A^{\#i}\owedge B^{\#i}\owedge \Gamma^{\#i}\Imp C^{\#i}$. By corollary \ref{corollary:cor1}, we get $A^{\#j}\owedge B^{\#j} \owedge \Gamma^{\#j}\Imp C^{\#j}$. Then by $(\wedge\mathrm{L})$, we get $(A^{\#j}\wedge B^{\#j})\owedge \Gamma^{\#j}\Imp C^{\#j}$. Hence $(A\wedge B)^{\#j}\owedge \Gamma^{\#j}\Imp C^{\#j}$. The case $(\vee\mathrm{L})$ is quite similar.

($\wedge\mathrm{R}$) Let the premises be $\Gamma\Imp A$ with weight $i_1$ and $\Gamma\Imp B$ with weight $i_2$, and the conclusion $\Gamma\Imp A\wedge B$ with weight $j$. Note that $i_1,i_2<j$. Assume $\Gamma^{\#i_1}\Imp A^{\#i_1}$ and $\Gamma^{\#i_2}\Imp B^{\#i_2}$. By corollary \ref{corollary:cor1}, we get $\Gamma^{\#j}\Imp A^{\#j}$ and $\Gamma^{\#j}\Imp B^{\#j}$. Hence by $(\wedge\mathrm{R})$, we get $\Gamma^{\#j}\Imp A^{\#j}\wedge B^{\#j}$ . Hence by ($\mathrm{W^1}$) and ($\mathrm{\imp R}$), we obtain $\Gamma^{\#j}\Imp \top^{j}\imp ( A^{\#j}\wedge B^{\#j})$ The cases $(\vee\mathrm{R})$ and $\imp\mathrm{R}$ are quite similar. Now Let us check the $\imp$-rules.

($\imp \mathrm{L}^1$) Let the premise be $p,B,\Gamma\Imp E$ with weight $i$, and the conclusion $p\imp B, p, \Gamma\Imp E$ with weight $j$. Note that $i<j$. Then
$(p,B,\Gamma\Imp E)^{\#i} = p\owedge B^{\#i}\owedge \Gamma^{\#i}\Imp E^{\#i}$, and
$(p\imp B,p,\Gamma\Imp E)^{\#j} = (\top^j\imp p)\imp B^{\#j}\owedge p\owedge \Gamma^{\#j} \Imp E^{\#j}$. Assume that $\vdash_{\mathrm{L_{RBL'}}} p\owedge B^{\#i}\owedge \Gamma^{\#i}\Imp E^{\#i}$.
By assumption and $(\owedge\mathrm{E})$ we get $B^{\#i}\owedge p\owedge \Gamma^{\#i}\Imp  E^{\#i}$.
Since $p\Imp \top^j\imp p$ is provable, we apply $(\imp\mathrm{L})$ to $p\Imp \top^j\imp p$ and $B^{\#i}\owedge p\owedge \Gamma^{\#i}\Imp E^{\#i}$, and get
$(p\odot ((\top^j\imp p)\imp B^{\#i}))\owedge p\owedge \Gamma^{\#i}\Imp E^{\#i}$.
By $(\mathrm{W^1})$, we get $((\top\odot p)\odot ((\top^j\imp p)\imp B^{\#i}))\owedge p\owedge \Gamma^{\#i}\Imp E^{\#i}$.
By proposition \ref{prop:1} (1) and (Cut), we obtain $ \top\odot (p\owedge (\top^j\imp p)\imp B^{\#i}))\Imp (\top\cdot p)\cdot ((\top^j\imp p)\imp B^{\#i})$.
Hence $(\top\odot (p\owedge (\top^j\imp p)\imp B^{\#i}))\owedge (p\owedge \Gamma^{\#i})\Imp E^{\#i}$.
By $(\mathrm{W^1})$, we get $(\top\odot (p\owedge (\top^j\imp p)\imp B^{\#i}))\owedge (\top\odot (p\owedge \Gamma^{\#i}))\Imp E^{\#i}$.
Hence by proposition \ref{prop:1} (3), we get $\top\odot (((p\owedge (\top^j\imp p)\imp B^{\#i}))\owedge (p\owedge \Gamma^{\#i}))\Imp E^{\#i}$.
By $(\imp\mathrm{R})$, we get $(p\owedge ((\top^j\imp p)\imp B^{\#i}))\owedge (p\owedge \Gamma^{\#i})\Imp \top \imp E^{\#i}$.
By $(\owedge\mathrm{C})$ and $(\owedge\mathrm{E})$, we get $(\top^j\imp p)\imp B^{\#i}\owedge p\owedge \Gamma^{\#i} \Imp \top \imp E^{\#i}$.
Finally, since $j\geq i+1$, by corollary \ref{corollary:cor1}, we get $(\top^j\imp p)\imp B^{\#j}\owedge p\owedge \Gamma^{\#j} \Imp  E^{\#j}$.

($\imp \mathrm{L}^2$) Let the premise be $C\imp(D\imp B),\Gamma\Imp E$ with weight $i$, and the conclusion $C\wedge D\imp B,\Gamma\Imp E$ with weight $j$. Then $(C\imp(D\imp B),\Gamma\Imp E)^{\#i}= (C^{\#i}\imp( D^{\#i}\imp B^{\#i}))\owedge \Gamma^{\#i}\Imp E^{\#i}$, and $((C\wedge D)\imp B),\Gamma\Imp E)^{\#i}=((\top^j\imp(C^{\#j}\wedge D^{\#j}))\imp B^{\#j})\owedge \Gamma^{\#j}\Imp E^{\#j}$. Assume that $\vdash_\mathrm{{L_{RBL'}}}(C^{\#i}\imp( D^{\#i}\imp B^{\#i}))\owedge \Gamma^{\#i}\Imp E^{\#i}$.
By assumption and corollary \ref{corollary:cor1}, we get $(C^{\#j}\imp( D^{\#j}\imp B^{\#j}))\owedge \Gamma^{\#j}\Imp E^{\#j}$. Hence by proposition \ref{prop:1} (9) and (Cut), we obtain $((\top^j\imp(C^{\#j}\wedge D^{\#j}))\imp B^{\#j})\owedge \Gamma^{\#j}\Imp E^{\#j}$.

($\imp \mathrm{L}^3$) Let the premise be $C\imp B, D\imp B,\Gamma\Imp E$ with weight $i$, and the conclusion $C\vee D\imp B,\Gamma\Imp E$ with weight $j$. Then $(C\imp B, D\imp B,\Gamma\Imp E)^{\#i}=( C^{\#i}\imp B^{\#i})\owedge(D^{\#i}\imp B^{\#i})\owedge \Gamma^{\#i}\Imp E^{\#i}$, and $(C\vee D\imp B,\Gamma\Imp E)^{\#j}=((\top^j\imp( C^{\#j}\vee D^{\#j}))\imp B^{\#j})\owedge \Gamma^{\#j}\Imp E^{\#j}$. Assume that $\vdash_\mathrm{{L_{RBL'}}} ( C^{\#i}\imp B^{\#i})\owedge(D^{\#i}\imp B^{\#i})\owedge \Gamma^{\#i}\Imp E^{\#i}$. By assumption and corollary \ref{corollary:cor1}, we get $( C^{\#j}\imp B^{\#j})\owedge(D^{\#j}\imp B^{\#j})\owedge \Gamma^{\#j}\Imp E^{\#j}$. Hence by proposition \ref{prop:1} (10) and (Cut), we obtain $((\top^j\imp( C^{\#j}\vee D^{\#j}))\imp B^{\#j})\owedge \Gamma^{\#j}\Imp E^{\#j}$.

($\imp \mathrm{L}^4$)  Let the premises be $D\imp B, C, \Gamma\Imp D$ with weight $i_1$ and $B, \Gamma\Imp E$ with weight $i_2$. Let the conclusion be $(C\imp D)\imp B,\Gamma,\Imp E$ with weight $j$. Suppose that $i_1,i_2<j$. Assume that $\vdash_\mathrm{{L_{RBL'}}} (D^{\#i_1}\imp B^{\#i_1})\owedge C^{\#i_1} \owedge \Gamma^{\#i_1}\Imp D^{\#i_1}$  and $\vdash_{\mathrm{L_{RBL'}}} B^{\#i_2}\owedge \Gamma^{\#i_2}\Imp E^{\#i_2}$.  It suffices to show that $\vdash_\mathrm{{L_{RBL'}}} ((\top^j\imp(C^{\#j}\imp D^{\#j}))\imp B^{\#j})\owedge\Gamma^{\#j}\Imp E^{\#j}$. Let us consider the first premise. By ($\owedge A^1$),($\owedge A^1$), ($\owedge E$), proposition \ref{prop:1} (5) and (Cut) we get $C^{\#i_1}\odot (( D^{\#i_1}\imp B^{\#i_1})\owedge\Gamma^{\#i_1})\Imp D^{\#i_1}$. Then by $(\imp\mathrm{R})$, we get $( D^{\#i_1}\imp B^{\#i_1})\owedge\Gamma^{\#i_1}\Imp C^{\#i_1}\imp D^{\#i_1}$. Hence by applying ($\mathrm{\imp L}$) to this resulting sequent and the second premise $B^{\#i_2}\owedge \Gamma^{\#i_2}\Imp E^{\#i_2}$, we obtain $((( D^{\#i_1}\imp B^{\#i_1})\owedge\Gamma^{\#i_1})\odot((C^{\#i_1}\imp D^{\#i_1})\imp B^{\#i_2})\owedge \Gamma^{\#i_2}\Imp E^{\#i_2}$.
By ($\mathrm{W^1}$), we get $(\top \odot ( D^{\#i_1}\imp B^{\#i_1}\owedge\Gamma^{\#i_1})\odot(C^{\#i_1}\imp D^{\#i_1})\imp B^{\#i_2})\owedge \Gamma^{\#i_2}\Imp E^{\#i_2}$. By Proposition \ref{prop:1} (1) and (Cut), $\top \odot (( D^{\#i_1}\imp B^{\#i_1}\owedge\Gamma^{\#i_1})\owedge ((C^{\#i_1}\imp D^{\#i_1})\imp B^{\#i_2}))\Imp (\top \cdot (( D^{\#i_1}\imp B^{\#i_1})\wedge\mu(\Gamma^{\#i_1}))\cdot((C^{\#i_1}\imp D^{\#i_1})\imp B^{\#i_2})$. So by ($\mathrm{\cdot R}$), ($\mathrm{\wedge L}$) and (Cut), we get $(\top \odot (( D^{\#i_1}\imp B^{\#i_1}\owedge\Gamma^{\#i_1})\owedge ((C^{\#i_1}\imp D^{\#i_1})\imp B^{\#i_2})))\owedge \Gamma^{\#i_2}\Imp E^{\#i_2}$. Again by ($\mathrm{W^1}$), we obtain $(\top \odot (( D^{\#i_1}\imp B^{\#i_1}\owedge\Gamma^{\#i_1})\owedge ((C^{\#i_1}\imp D^{\#i_1})\imp B^{\#i_2})))\owedge (T\odot \Gamma^{\#i_2})\Imp E^{\#i_2}$. By Proposition \ref{prop:1} (3), (Cut), ($\mathrm{\owedge A^1}$) and ($\mathrm{\owedge A^2}$), we get $ T \odot ((D^{\#i_1}\imp B^{\#i_1})\owedge\Gamma^{\#i_1}\owedge ((C^{\#i_1}\imp D^{\#i_1})\imp B^{\#i_2})\owedge \Gamma^{\#i_2})\Imp E^{\#i_2}$. So by ($\mathrm{\imp R}$), we obtain $ (D^{\#i_1}\imp B^{\#i_1})\owedge\Gamma^{\#i_1}\owedge( (C^{\#i_1}\imp D^{\#i_1})\imp B^{\#i_2})\owedge \Gamma^{\#i_2}\Imp \top \imp E^{\#i_2}$. Since $j\geq i_2+1, i_1+1$, by corollary \ref{corollary:cor1}, we get $ D^{\#j}\imp B^{\#j}\owedge\Gamma^{\#j}\owedge (C^{\#j}\imp D^{\#j})\imp B^{\#j}\owedge \Gamma^{\#j}\Imp E^{\#j}$. Hence by ($\mathrm{\owedge E}$),  ($\mathrm{\owedge C}$), and ($\mathrm{\wedge L}$), we obtain $ (D^{\#j}\imp B^{\#j})\wedge ((C^{\#j}\imp D^{\#j})\imp B^{\#j})\owedge \Gamma^{\#j}\Imp E^{\#j}$. Finally by by proposition \ref{prop:1} (10) and (Cut), we get $((\top^j\imp(C^{\#j}\imp D^{\#j}))\imp B^{\#j})\owedge\Gamma^{\#j}\Imp E^{\#j}$.
\end{proof}

By lemma \ref{lemma:rblequiv} and theorem \ref{thm:g4ip}, we get the following theorem.

\begin{theorem}\label{thm:g4ip}
For any  $\mc{L}_{\mathrm{Int}}$-sequent $\Gamma\Imp A$, if $\vdash_{\mathsf{G4ip}}\Gamma\Imp A$, then $\vdash_{\mathsf{L_{RBL}}}\Gamma\Imp A)^\#$.
\end{theorem}
For any $\mc{L}_{\mathrm{Int}}$-formula $A$ and an occurrence of its subformula $B$, define $A\{B/\top^n\imp B\}$ as the formula obtained from $A$ by replacing this occurrence of $B$ by $\top^n\imp B$.
\begin{lemma}\label{lemma:int0}
For any $\mc{L}_{\mathrm{Int}}$-formula $A$ and an occurrence of its subformula $B$, $\vdash_{\mathsf{G4ip}} A\{B/\top^n\imp B\}\Leftrightarrow A$
\end{lemma}
\begin{proof}
By induction on the complexity of $A$.

Case 1. $A=p$ for some propositional letter $p$. It is easy to see that $\vdash_{\mathsf{G4ip}}\top^n\imp p\Leftrightarrow p$.

Case 2. $A=A_1\imp A_2$. If $B=A$ then obviously we have $\vdash_{\mathsf{G4ip}} \top^n\Imp A \Leftrightarrow A$. Otherwise $B$ occurs in $A_1$ or $A_2$. Assume that $B$ occurs in $A_1$. Then by induction hypothesis $\vdash_{\mathsf{G4ip}} A_1\{B/\top^n\imp B\} \Leftrightarrow A_1$. Hence by ($\imp L$) and ($\imp R$), we get $\vdash_{\mathsf{G4ip}} A_1\{B/\top^n\imp B\}\imp A_2 \Leftrightarrow A_1\imp A_2$. The case that $B$ occurs in $A_2$ is similar.

Case 3. $A=A_1\wedge A_2$ or $A=A_1\vee A_2$. The proof is similar to case 2.
\end{proof}
Since formula $Tr(A)$ is obtained from formula $A$ by replacing some occurrences of subformula $B$ by $\top^n \imp B$ for some $n\geq 0$, by lemma \ref{lemma:int0}, we get the following corollary immediately.
\begin{corollary}\label{lemma:int}
For any $\mc{L}_{\mathrm{Int}}$-formula $A$, $\vdash_{\mathsf{G4ip}}Tr(A)\Leftrightarrow A$.
\end{corollary}

\begin{theorem}\label{thm:rblemb}
For any $\mc{L}_{\mathrm{Int}}$-formula $A$, $\vdash_{\mathsf{G4ip}}\Imp A$ iff $\vdash_{\mathsf{L_{RBL}}}\top\Imp Tr(A)$.
\end{theorem}
\begin{proof}
The left-to-right direction follows from theorem \ref{thm:g4ip}. For the other direction, Assume $\vdash_{\mathsf{L_{RBL}}}\top\Imp Tr(A)$. Since $\mathsf{L_{RBL}}$ is a conservative extension of $\mathsf{BPL}$ (\cite{ML14}),
we obtain $\vdash_{\mathsf{BPL}} Tr(A)$. Since $\mathsf{BPL\sub Int}$, we get $\vdash_{\mathsf{Int}} Tr(A)$.
Then $\vdash_{\mathsf{G4ip}}\Imp Tr(A)$. By corollary \ref{lemma:int}, we get $\vdash_{\mathsf{G4ip}}\Imp A$.
\end{proof}

The following theorem follows immediately from theorem \ref{thm:rblemb} and \ref{theorem:conservative1}.

\begin{theorem}\label{thm:emb}
For any $\mc{L}_{\mathrm{Int}}$-formula $A$, $\vdash_{\mathsf{Int}} A$ iff $\vdash_{\mathsf{BPL}} Tr(A)$.
\end{theorem}

It is well-known that $\mathsf{Int}$ is embedded into the modal logic $\mathsf{S4}=K\oplus \B p\imp p\oplus \B p\imp\B \B p$ by G\"{o}del's translation (\cite{GO33,MT48}) $\mathsf{G}$ which is defined recursively as follows: $\mathsf{G}(p) = \B p$; $\mathsf{G}(\bot)=\bot$; $\mathsf{G}(A\wedge B)=\mathsf{G}(A)\wedge \mathsf{G}(B)$; $\mathsf{G}(A\vee B)=\mathsf{G}(A)\vee \mathsf{G}(B)$; $\mathsf{G}(A\imp B)=\B(\mathsf{G}(A)\imp \mathsf{G}(B))$.
Zakharyaschev (\cite{Zak97}) proved that the modal logic $\mathsf{Grz}=\mathsf{K}\oplus \B(\B( p\imp \B p))\imp p$ is the greatest extension of $\mathsf{S4}$ which intuitionistic logic can be embedded into. Esakia proved that the modal logic $\mathsf{S4}$ is embeddable into the modal logic $\mathsf{wK4} = \mathsf{K}\oplus p\wedge\B p\imp\B\B p$ (\cite{Esa01, Esa04}) by the translation $\mathsf{Sp}$, the mapping of the set of modal formulae into itself, commuting with Boolean connectives and $\mathsf{Sp}(\D p) = p\vee\D p$ and $\mathsf{Sp}(\B p) = p\wedge\B p$.
Hence $\mathsf{Int}$ is embedded into $\mathsf{wK4}$ via the composition $\mathsf{Sp}\circ\mathsf{G}$.
\begin{center}
\includegraphics{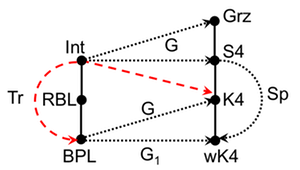}
\end{center}
Moreover it is known that Visser's basic propositional logic $\mathsf{BPL}$ is embedded into modal logic $\mathsf{K4}$ via G\"{o}del's translation $\mathsf{G}$ (\cite{visser81}). It is also known that $\mathsf{BPL}$ is embedded into $\mathsf{wK4}$ by the variant of $\mathsf{G}$ denoted by $\mathsf{G_1}$ which sends each propositional letter $p$ to $p\wedge \B p$ (\cite{SM13}).
By the theorem \ref{thm:emb}, we get the following new results: $\mathsf{Int}$ is embedded into $\mathsf{K4}$ by the map $\mathsf{G}\circ\mathsf{Tr}$; and $\mathsf{Int}$ is embedded into $\mathsf{wK4}$ by the map $\mathsf{G_1}\circ\mathsf{Tr}$.

By Ladner \cite{Lad77} results, we know that modal logic K4 is PSPACE complete. By Visser's translation it trivially follows that
 BPL is in PSPACE by the G$\ddot{o}$del translation. Note that our translation is a polynomial time tranlation. Consequently since intuitionistic logic is PSPACE complete \cite{Sta79} (intuitionistic logic logic is PSPACE complete), by theorem \ref{thm:emb}, we obtain that BPL is PSPACE-hard. Hence we get the following corollary.
\begin{corollary}
The logic BPL is PSPACE complete.
\end{corollary}
This complexity result was first proved by Bou in \cite{bou04} via a polynominal time reduction from $\mathsf{QBF}$ to $\mathsf{BPL}$.
However, our proof of PSPACE completeness differs from it.

\bibliographystyle{plain}
\bibliography{rbl2}

\paragraph{Acknowledgements.}
 The first author was supported by the project of China National Social Sciences Fund (Grant no. 12CZX054).

\end{document}